 \documentclass[12pt]{article}

\usepackage[letterpaper, hmargin=1in, top=1in, bottom=1.2in, footskip=0.7in]{geometry}

\usepackage{rotating,pdflscape,xcolor,amssymb,subfigure,psfrag,amsmath,eufrak,bbm,epsfig,amsthm,mathtools,fancyhdr,graphicx}
\definecolor{vdarkred}{rgb}{0.6,0,0.2}
\definecolor{vdarkblue}{rgb}{0,0.2,0.6}
\usepackage[pdftex, colorlinks, linkcolor=vdarkblue,citecolor=vdarkred]{hyperref}
\usepackage{a4wide,amscd}
\usepackage{tikz} 
 \usepackage[usenames,dvipsnames]{pstricks}
 \usepackage{pst-grad} 
 \usepackage{pst-plot} 

\usepackage[small]{caption}

\makeatletter
\def\z@first#1#2{#1}
\def\z@second#1#2{#2}
\def\z@zp@selectchar#1#2{
	\IfStrEqCase{#2}{%
		{p}{#1{(}{)}}%
		{P}{#1{)}{(}}%
		{c}{#1{[}{]}}%
		{C}{#1{]}{[}}%
		{a}{#1{\{}{\}}}%
		{A}{#1{\}}{\{}}%
		{i}{#1{[}{]}\!#1{[}{]}}%
		{I}{#1{]}{[}\!#1{]}{[}}%
		{t}{#1{<}{>}}%
		{T}{#1{>}{<}}%
		{b}{#1{|}{|}}%
		{n}{#1{\|}{\|}}%
		{v}{#1{.}{.}}%
	}[#1{(}{)}]%
}
\def\z@zp#1#2\fin#3{
	\z@zp@selectchar{\left\z@first}{#1}#3
	\zifempty{#2}%
	{\z@zp@selectchar{\right\z@second}{#1}}%
	{\z@zp@selectchar{\right\z@second}{#2}}%
}
\newcommand{\zp}[2][p]{\zifempty{#1}{\left(#2\right)}{\z@zp#1\fin{#2}}}
\makeatother

\usepackage{xstring}
\def\zifempty#1#2#3{\def\foo{#1}\ifx\foo\empty\relax#2\else#3\fi}

\newcommand{\lam}{\lambda}

\newcommand{\vfi}{\varphi}
\newcommand{\al}{\alpha}

\newcommand{\del}{\delta}

\newcommand{\ld}{\ldots}

\newcommand{\bgt}{\begin{itemize}}
\newcommand{\ent}{\end{itemize}}

\newcommand{\ds}{\displaystyle}

\newcommand{\brem}{\begin{rmk}}
\newcommand{\erem}{\end{rmk}}
\newcommand{\blem}{\begin{lem}}
\newcommand{\elem}{\end{lem}}
\newcommand{\bcor}{\begin{cor}}
\newcommand{\ecor}{\end{cor}}
\newcommand{\bTh}{\begin{Th}}
\newcommand{\eTh}{\end{Th}}
\newcommand{\bpropo}{\begin{propo}}
\newcommand{\epropo}{\end{propo}}

\newcommand{\op}{\operatorname}

\newcommand{\ninf}{\underset{n\to\infty}{\longrightarrow}}

\newcommand{\E}{\op{\mathbb{E}}}

\newcommand{\R}{\mathbb{R}}

\newcommand{\p}{\mathbb{P}}

\newcommand{\f}{\frac}
\newcommand{\ff}{\frac{1}}

\newcommand{\ste}{\, ;\, }

\newcommand{\bbm}{\begin{bmatrix}}
\newcommand{\ebm}{\end{bmatrix}}
\newcommand{\bes}{\begin{equation*}}
\newcommand{\ees}{\end{equation*}}
\newcommand{\be}{\begin{equation}}
\newcommand{\ee}{\end{equation}}
\newcommand{\beqy}{\begin{eqnarray}}
\newcommand{\eeqy}{\end{eqnarray}}
\newcommand{\beq}{\begin{eqnarray*}}
\newcommand{\eeq}{\end{eqnarray*}}

\newcommand{\bpm}{\begin{pmatrix}}
\newcommand{\epm}{\end{pmatrix}}

\newcommand{\cd}{\cdots}

\newtheorem{Th}{Theorem}[section]

\newtheorem{propo}[Th]{Proposition}
\newtheorem{Prop}[Th]{Proposition}

\newtheorem{lem}[Th]{Lemma}

\newtheorem{cor}[Th]{Corollary}

\theoremstyle{definition}
\newtheorem{rmk}[Th]{Remark}
\newtheorem{Def}[Th]{Definition}

\long\def\symbolfootnote[#1]#2{\begingroup
\def\thefootnote{\fnsymbol{footnote}}\footnote[#1]{#2}\endgroup} 
\providecommand{\keywords}[1]
{
	\small	
	\textbf{\textit{\footnotesize Keywords ---}} #1
}

\makeatletter
\def\@addpunct#1{%
	\relax\ifhmode
	\ifnum\spacefactor>\@m \else#1\fi
	\fi}
\newcommand{\keywordsname}{Key words}
\def\@setkeywords{%
	{\itshape \keywordsname.}\enspace \@keywords\@addpunct.}
\def\keywords#1{\def\@keywords{#1}}
\let\@keywords=\@empty
\g@addto@macro{\maketitle}{\begingroup%
	\let\@makefnmark\relax  \let\@thefnmark\relax%
	\ifx\@keywords\@mpty\else\@footnotetext{\@setkeywords}\fi%
	\endgroup}
\makeatletter
\def\blfootnote{\gdef\@thefnmark{}\@footnotetext}
\usepackage[affil-it]{authblk}
\date{}
\author{Cambyse Pakzad} 
\title{Extremes of Chi triangular array from the Gaussian $\beta$-Ensemble at high temperature}

\newcommand{\Addresses}{{
		\bigskip
		\footnotesize
		\textsc{MAP 5, UMR CNRS 8145 - Universit\'e Paris Descartes, France}\par\nopagebreak
		\textit{E-mail address}: \texttt{cambyse.pakzad@gmail.com} 
	}}
		\keywords{Gaussian $\beta$-Ensembles, Poisson statistics,  Extreme value theory.}
\begin{document}
	\maketitle
\blfootnote{\textup{2010} \textit{Mathematics Subject Classification}.
	60B20 - 60F05 - 60G70}
\begin{abstract}We study the extreme point process associated to the off-diagonal components in the matrix representation \cite{MatrixModelBetaEnsemble} of the Gaussian $\beta$-Ensemble and prove its convergence to Poisson point process as $n\to +\infty$ when the inverse temperature $\beta$ scales with $n$ and tends to $0$. We consider two main high temperature regimes: $\ds{\beta\ll \ff{n}}$ and $\ds{n\beta= 2\gamma \geq 0}$. The normalizing sequences are explicitly given in each cases. As a consequence, we estimate the first order asymptotic of the largest eigenvalue of the Gaussian $\beta$-Ensemble.
\end{abstract}
\section{Introduction}
Gaussian ensembles, namely Hermitian matrices with independent Gaussian entries whose joint distribution invariant is under conjugation by appropriate unitary matrices, play a central role in the realm of Random Matrix Theory. Regarding their spectrum, the joint distribution of the eigenvalues is available through the formula:
\begin{align}\label{def_beta_ens}
P_{n,\beta}(d\lam_1,...,d\lam_n):= \ff{Z_{n,\beta}} \exp\left( -\frac{\beta}{4} \sum_{i=1}^{n}\lam^2_i\right)  \prod_{i<j}^{n}\left| \lam_j-\lam_i\right| ^\beta \prod_{i=1}^{n}\mathrm{d}\lam_i.
\end{align}
According to the Dyson index $\beta$, the density (\ref{def_beta_ens}) describes the eigenvalues of the Gaussian Orthogonal Ensemble $(\beta=1)$, the Gaussian Unitary Ensemble $(\beta=2)$, or the Gaussian Symplectic Ensemble $(\beta=4)$. Although it seemed that no natural matrix model could be associated to the Gaussian $\beta$-Ensemble, a collection of $n$ points with distribution (\ref{def_beta_ens}) and general $\beta>0$, the gap is bridged many years later by Edelman and Dumitriu in \cite{MatrixModelBetaEnsemble}. Through the Househölder method, they constructed the following tridiagonal symmetric random matrix:
\begin{align}\label{H}
 H_{n,\beta}:=\ff{\sqrt{\beta}} \begin{pmatrix}
\mathcal{N}(0,2) &\chi\left( \left( n-1\right) \beta\right)  & && \\
\chi\left( \left( n-1\right) \beta\right) &\mathcal{N}(0,2) &\chi\left( \left( n-2\right) \beta\right) &&\\
&\chi\left( \left( n-2\right) \beta\right) &\mathcal{N}(0,2) & \chi\left( \left( n-3\right) \beta\right) &  \\
& & \ddots &\ddots & \ddots &    \\
& &\qquad  \ddots & \qquad \ddots &\chi\left(  \beta\right)  &   \\
& & &\chi\left( \beta\right) & \mathcal{N}(0,2) 
\end{pmatrix} 
\end{align} with independence between all variables, such that for any $n\geq 1$ and $\beta> 0$, the eigenvalues $(\lam_1,...,\lam_n)$ of $H_{n,\beta}$ have exactly joint distribution $P_{n,\beta}$ from (\ref{def_beta_ens}).

Let us consider the off-diagonal sequence $\left( X_{i,n}\right)_{i\leq n} $ of $H_{n,\beta}$, namely  $\left( X_{i,n}\right) $ are independent random variables with distribution $\ds{X_{i,n}\sim \chi\left( i\beta\right) }$ for $\ds{i\leq n}$. 

The scope of this paper is to study the asymptotic behavior of the largest $X_{i,n}$, namely of the extreme point process $\ds{\sum_{i=1}^{n}\delta_{a_n\left( X_{i,n}-b_n\right)}}$ where $(a_n)$, $(b_n)$ are suitably chosen normalizing sequences and when the inverse temperature $\beta:=\beta_n$ scales with $n$ in such a way that $\beta \xrightarrow[n\infty]{} 0$. We consider two regimes in this context. In first place, we study the case $\ds{n^{-2}\ll \beta \ll {n}^{-1}}$ which itself divides into two subregimes around $\ds{\left( n\log\log n\right)^{-1} }$. Then, the intermediate high temperature regime $\ds{n\beta =2\gamma}$ is analyzed for $\gamma \in (0,1)$. Both cases lead to an inhomogeneous limiting Poisson point process on the half line $\R^+$ with intensity measure of exponential kind. The point process approach is borrowed from the Extreme Value Theory (see \cite{Leadbetter,Resnick}) as it shows to be a convenient way to encompass many informations. In this vein, our results are stated in terms of Poisson convergence. One consequence of this fact is that one can derive the limiting distribution of the rescaled largest component to be Gumbel. We exploit here the quantitative growth rate of the largest entries in (\ref{H}) to provide an estimate on the first order asymptotic of the largest eigenvalue in the high temperature regime.

The question of the normalizing (or scaling) sequences is fundamental. They indicate the precise growth rate of the maximum of the set of random variables considered, which, roughly speaking, provides the right way to look at the extreme values. Although they are not always explicit or with closed form, we furnish here exact formulae. The case $\ds{n\beta =2\gamma}$ has a nice interpretation: the scaling sequences for $\ds{\max_{1\leq i \leq n}X_{i,n}}$ thoroughly match with those of the maximum of $\ds{\frac{n}{\gamma \log\log(n)}}$ i.i.d. copies of $X_{n,n}\sim \chi(2\gamma)$.

We now state our main result in the case $\ds{n^{-2}\ll \beta \ll {n}^{-1}}$:
\begin{Th}
	\label{Main_th}
	Let $\beta := \beta_n \ll 1$ such that $\ds{\ff{n^2}\ll \beta \ll \ff{n } }$ and $\ds{X_{i,n}\sim \chi \left( i\beta\right)  }$, with $1\leq i \leq n$, a triangular array of independent random variables. 
	
	Assume $\ds{\ff{n^2}\ll \beta \ll \ff{n\log \log \left( n\right) }  }$. Let the scaling sequences: \begin{align}\label{defTH_scaling_1}
	a_n={\sqrt{2\log \left( n^2 \beta \right) }},\qquad b_n=\sqrt{2\log \left( n^2 \beta \right) } - \frac{\log \log\left(  n^2 \beta\right)  }{\sqrt{2\log \left( n^2 \beta \right) }}.
	\end{align}
	Then the point process $\ds{\sum_{i=1}^{n}\delta_{a_n\left( X_{i,n}-b_n\right)}}$ converges weakly to an inhomogeneous Poisson point process on $\R^+$ with intensity measure $\ds{\frac{e^{-x}}{4}\mathrm{d}x}$.
	
	Assume $\ds{ \ff{n\log \log \left( n\right) } \ll \beta \ll \ff{n} }$. Let the scaling sequences:
	\begin{align}\label{defTH_scaling_2}a_n=\sqrt{2\log \left( n \right) },\qquad b_n=a_n - \frac{n\beta \log \log\left(  n\right)  }{\sqrt{2\log \left( n \right) }}-\frac{\log \log  \left( n\right) }{\sqrt{2\log \left( n \right) }}-\frac{\log \log \log \left( n\right) }{\sqrt{2\log \left( n \right) }}.\end{align}Then the point process $\ds{\sum_{i=1}^{n}\delta_{a_n\left( X_{i,n}-b_n\right)}}$ converges weakly to an inhomogeneous Poisson point process on $\R^+$ with intensity measure $\ds{e^{-x}\mathrm{d}x}$.
\end{Th}
In the intermediate high regime $\ds{n\beta=2\gamma}$, we have the following result:
\begin{Th}
	\label{Main_th2}
	Let $\ds{\gamma \in (0,1)}$ and $\ds{ \beta_n = {2\gamma}n^{-1}}$. Let $\ds{X_{i,n}\sim \chi \left( i\beta_n\right)  }$, with $1\leq i \leq n$, a triangular array of independent random variables. Let the scaling sequences: \begin{align}\label{defTH_scaling_3}
	a_n:={\sqrt{2\log ( \frac{n}{\gamma\log \log(n)})}},\quad  b_n:=a_n + \frac{ \left( \frac{\gamma}{2}-1\right) \log\log(\frac{n}{\gamma\log \log(n)})-\log\left( 2^{-\frac{\gamma}{2}}\Gamma(\gamma)\right) }{\sqrt{2\log ( \frac{n}{\gamma\log \log(n)})}}.
	\end{align}
	Then the point process $\ds{\sum_{i=1}^{n}\delta_{a_n\left( X_{i,n}-b_n\right)}}$ converges weakly to an inhomogeneous Poisson point process on $\R^+$ with intensity measure $\ds{{e^{-x}}{}\mathrm{d}x}$. 
\end{Th}
\begin{rmk}
On one hand, we consider a size sample of $(X_{i,n})$ growing in $n$ and on the other hand, the greatest chi parameter of $(X_{i,n})$ is $n\beta$. Assuming that $n\beta\ll 1$, one could expect that the maximum of the $(X_{i,n})$ should converge to $0$. Our result shows the other way. In the regime $\ds{n^{-2}\ll \beta\ll n^{-1}}$, the centering $(b_n)$ is still diverging to infinity. Though, one can retrain our work to show it is converging to $0$ when $\ds{\beta\ll n^{-2}}$. We did not consider this special case here since the competition against the maxium of Gaussian variables becomes irrelevant.
\end{rmk}
Concerning the edge statistics of $H_{n,\beta}$, we derive the following corollary:
\begin{cor} Let $\lam_{\max}\left( H\right) $ the largest eigenvalue of $\ds{H:=\sqrt{\beta}H_{n,\beta}}$ defined in (\ref{H}). Assume the hypothesis of Theorem \ref{Main_th} or \ref{Main_th2}. There exists constants $0<c<c'<+\infty$ such that with probability tending to $1$, as $n\to +\infty$, 
	$$\frac{\lam_{\max}\left(H \right)}{\sqrt{\log(n)}} \in [c,c']. $$
\end{cor}
\begin{proof} Consider the standard basis $e_i:=(\delta_{i,j})$ of $\R^n$. The characterization of the largest eigenvalue with the Rayleigh quotient (see Theorem 4.2.2 in \cite{Horn}) allows us to write:
\begin{align}\label{corollary_lowerbound}
\lam_{\max}\left( H \right) = \max_{x\neq 0} \frac{x^t Hx}{x^t x}\geq \frac{e_i^t He_i}{e_i^t e_i} =H\left( {i,i}\right), \quad 1\leq i\leq n .
\end{align}		
		Thus we get the lower bound $\ds{\lam_{\max}\left( H\right)\geq \max_{1\leq i \leq n}H\left(i,i \right) }$.
		
		 On the other hand, we apply the Gerschgorin circles (see Theorem 6.1.1 in \cite{Horn}) to get an upperbound: 
		\begin{align}\label{corollary_upperbound}
 \lam_{\max}\left( H\right)\leq \max_{1\leq i \leq n}H\left(i,i \right) + 2\max_{1\leq i \leq n-1}H\left(i,i+1 \right).
		\end{align}
		The claim yields by combining both bounds (\ref{corollary_lowerbound}), (\ref{corollary_upperbound}) with the conclusions of Theorem \ref{Main_th} or \ref{Main_th2}, and the counterpart result for maximum of i.i.d. Gaussian variables, that is the point process $\ds{\sum_{i=1}^{n-1}\delta_{a_n\left(H\left(i,i+1 \right)-b_n\right)}}$ converges weakly in probability to a Poisson point process, where the normalizing sequences are: $$b_n = 2\sqrt{\log(n)}-\frac{\log\log(n)+\log(4\pi)}{2\sqrt{\log(n)}},\quad a_n=2\sqrt{\log(n)}.$$
\end{proof}
In Section \ref{section_Multivariate}, we establish a general tool for asymptotic Poissonian statistics of extreme point processes in the independent but non identically distributed context. Theorem \ref{Th_tool} will be the framework of our method and is the analog of the classic Poisson limit theorem for i.i.d. extremes (see for instance Proposition 3.21 in \cite{Resnick}). Sections \ref{section_proof1} and \ref{section_proof2} are devoted to the proof of each case following the steps accordingly to Theorem \ref{Th_tool}, namely we prove appropriate convergence and uniform negligibility.
\section{Multivariate Poisson limit}\label{section_Multivariate}
For completness, we recall the definition of a Poisson point process. 
\begin{Def}\label{def_ppp}
	Let $X$ a topological space endowed with a $\sigma$-finite measure $\mu$. A Poisson point process on $X$ with intensity $\mu$ is a random Radon measure $N$ on $X$ such that for any $k\geq 1$ and $B_1,\ldots,B_k$ pairwise disjoint Borel sets of $X$ with finite $\mu$ measure, the random vector $\left( N\left( B_1\right) ,\ldots, N\left( B_k\right) \right) $ has law $\ds{\op{Po}\left( \mu(B_1\right)\otimes \cdots \otimes \op{Po}\left( \mu(B_k\right) }$.
\end{Def}
In order to reach the point process level, we need the following finite-dimensional Poisson approximation:
\begin{Prop}\label{multivariate_poisson_limit_prop}
	Let $(X_{i,n})$ a triangular array of independent random variables with same interval support $\mathcal{S}$. Let $I\subset \R$ an interval and $\ds{\vfi_n:I\to \mathcal{S}}$ an increasing homeomorphism. Let $k\geq 1$ and $ x_0<\ldots <x_k\in I$ and the random vector $(N_0,\ld, N_k)$ defined for $ \ell \leq k$ by: $$N_\ell=\#\{ i=1, \ld, n\ste X_{i,n}\in (\vfi_n(x_{\ell-1}),\vfi_n(x_{\ell})]\}= \sum_{i=1}^n \delta_{\{ X_{i,n}\in (\vfi_n(x_{\ell-1}),\vfi_n(x_{\ell})] \}}.$$ 
	For $1\leq \ell \leq k<+\infty$, we suppose that: \begin{align}\label{condition_for_poisson1}
	\sum_{i=1}^n\p\left( X_{i,n}\in (\vfi_n(x_{\ell-1}),\vfi_n(x_{\ell})]\right) \xrightarrow[n\infty]{} \lam_\ell,
	\end{align}
	\begin{align}\label{condition_for_poisson2}
	\max_{1\leq \ell \leq k} \max_{i\leq n}\left( \p\left(X_{i,n}\in (\vfi_n(x_{\ell-1}),\vfi_n(x_{\ell})] \right) \right) \xrightarrow[n\infty]{}0.
	\end{align}
	
	Then we have the convergence in distribution: $$(N_1, \ld N_k)\ninf \op{Po}(\lam_1)\otimes\cd\otimes \op{Po}(\lam_k).$$ 
\end{Prop}
\begin{proof}
	For convenience, we set $\ds{p^{(n)}_{i,\ell}:=\p\left( X_{i,n}\in (\vfi_n(x_{\ell-1}),\vfi_n(x_{\ell})] \right) }$. 
	
	Let $t=(t_1,..,t_k) \in \R^k$, $N = (N_1,..,N_k)$. We compute the Laplace transform of $N$: \begin{align*} \E \big( e^{-<t,N>} \big) &=\E  \big( e^{-\sum_{\ell=1}^k \sum_{i=1}^n t_\ell 1_{\{  X_{i,n}\in (\vfi_n(x_{\ell-1}),\vfi_n(x_{\ell})]\} }} \big) \\ &=  \prod_{i=1}^n \E \big( e^{-\sum_{\ell=1}^k t_\ell 1_{\{ X_{i,n}\in (\vfi_n(x_{\ell-1}),\vfi_n(x_{\ell})]\} }} \big) \qquad \text{ by independence} \\ &= \prod_{i=1}^n \Big( \sum_{\ell=0}^k p^{(n)}_{\ell,i} e^{-t_{\ell}}\Big) \qquad \text{ by transfer and with $t_0 = 0$} \\ &=\prod_{i=1}^n \Big(1+ \sum_{\ell=1}^k p^{(n)}_{\ell,i}( e^{-t_{\ell}}-1)\Big) \\  &= \prod_{i=1}^n \Big(1+ a_{i}^{(n)} \Big) \qquad \text{ with $a_{i}^{(n)} := \sum_{\ell=1}^k p^{(n)}_{\ell,i}( e^{-t_{\ell}}-1)$}.  \end{align*}
	Taking the logarithm and Taylor-expanding, \begin{align*} \log  \E \big( e^{-<t,N>} \big)  &= \sum_{i=1}^n \log \Big(1+ \sum_{\ell=1}^k p_{\ell,i}^{(n)}( e^{-t_{\ell}}-1)\Big) \\ &=  \sum_{i=1}^n \log \Big(1+ a_{i}^{(n)}\Big) \\ &=\sum_{i=1}^n \Big( a_{i}^{(n)} - \f{(a_{i}^{(n)})^2}{2(1+c)^2} \Big) \qquad \text{ with $0<c=c_{i,n}<a_{i}^{(n)}$} \\ &= \sum_{i=1}^n  a_{i}^{(n)} - \sum_{i=1}^n \f{(a_{i}^{(n)})^2}{2(1+c)^2}  \end{align*} From hypothesis (\ref{condition_for_poisson1}) and (\ref{condition_for_poisson2}), we deduce: $$\ds{\sum_{i=1}^n a_{i}^{(n)} \xrightarrow[n\infty]{} \sum_{\ell=1}^k \lam_\ell (e^{-t_\ell}-1)},$$ $$\ds{0 \leq \sum_{i=1}^n\f{(a_{i}^{(n)})^2}{2(1+c)^2} \leq \sum_{i=1}^n (a_{i}^{(n)})^2 \leq \big( \max_{i \leq n} a_{i}^{(n)} \big)\Big( \sum_{i=1}^n a_{i}^{(n)} \Big) \xrightarrow[n\infty]{} 0. }$$ It follows that: \begin{align*} \log  \E \big( e^{-<t,N>} \big)  &=  \sum_{i=1}^n  a_{i}^{(n)} + o(1) \xrightarrow[n\infty]{} \sum_{\ell=1}^k \lam_\ell (e^{-t_\ell}-1).  \end{align*}
\end{proof}
We are now ready to show the goal of this section:
\begin{Th}\label{Th_tool}
	Let $(X_{i,n})$ a triangular array of independent random variables with same interval support $\mathcal{S}$. Assume there exists $I\subset \R$ interval, $\ds{G:I\to \R^+}$ c\`{a}dl\`{a}g with finite limit at the right boundary of $I$ and $\ds{\vfi_n:I\to \mathcal{S}}$ increasing homeomorphism such that: \begin{align}\label{condition_for_poisson3}
	\forall x\in I,\qquad \sum_{i=1}^n \p\left( X_{i,n}\geq \vfi_n\left( x\right) \right) \xrightarrow[n\infty]{} G(x),\end{align} and for $x,y\in I$ such that $y<x<+\infty$, \begin{align}\label{condition_for_poisson4}
	\max_{i \leq n} \p\left(  \vfi_n\left(y\right)  < X_{i,n} \leq \vfi_n\left( x\right)   \right) \xrightarrow[n\infty]{} 0.
	\end{align} Then the point process $$\sum_{i=1}^n \del_{\vfi^{-1}_n(X_{i,n})}$$ converges to a Poisson point process on the interval $I$ with intensity $\mu$ defined by $$\ds{\mu\left((x,y]  \right):=G(x)-G(y), \qquad x<y .}$$
\end{Th}
\begin{proof}
	Let $\ds{N_n:=\sum_{i=1}^n \del_{\vfi^{-1}_n(X_{i,n})}}$. By Definition \ref{def_ppp}, we need to show the following convergence in distribution, for any $k\geq 1$ and $x_0<\ldots<x_k\in I$, $$\left( N_n (x_0,x_1] , \ld N_n(x_{k-1},x_k]\right) \xrightarrow[n\infty]{} \op{Po}(\mu(x_0,x_1])\otimes\cd\otimes \op{Po}(\mu(x_{k-1},x_k])$$
	where for $\lam\ge 0$, $$\op{Po}(\lam):=e^{-\lam}\sum_{i\ge 0}\f{\lam^i}{i!}\del_i.$$
	First thing to notice is the following identity for $1\leq \ell\leq  k$, $$\ds{N_n (x_{\ell-1},x_\ell]=\#\{ i=1, \ld, n\ste X_{i,n}\in (\vfi_n(x_{\ell-1}),\vfi_n(x_{\ell})]\}.}$$ Thus, we aim to apply Proposition \ref{multivariate_poisson_limit_prop} in order to complete the proof. Clearly, the condition (\ref{condition_for_poisson2}) is fulfilled thanks to the uniform negligibility hypothesis (\ref{condition_for_poisson4}). Besides, for any $x\in I$, one has $\ds{\mu\left( \left( x,+\infty\right) \cap I \right) <+\infty}$ because of the hypothesis on $G$. Hence, using (\ref{condition_for_poisson3}), we get: $$\forall x<y\in I,\qquad \sum_{i=1}^n \p\left(\vfi_n(x)< X_{i,n}\leq \vfi_n\left( y\right) \right) \xrightarrow[n\infty]{} G(x) -G(y).$$ The assumption (\ref{condition_for_poisson1}) readily follows with $\ds{\lam_\ell = \mu\left((x_{\ell-1},x_\ell]  \right)=G(x_{\ell-1})-G(x_\ell)}$.
\end{proof}
\section{Proof of Theorem \ref{Main_th}}\label{section_proof1}
Before going in the details of the demonstration, let us fix our notation and recall some well-known facts about Gamma function and Gamma distribution.
\begin{rmk}\label{gamma_chi_relation_rmk}
	A real random variable $X$ follows a $\Gamma(\al,\beta)$ distribution, ie: $X\sim \Gamma(\al,\beta)$, if its density is $\ds{f_X(x)= \frac{\beta^\al x^{\al-1}e^{-\beta x}}{\Gamma(\al)}}$, $\ds{x\in (0,+\infty)}$. Also, $$ \Gamma(a,z):= \p\left(\Gamma\left( a,1\right)  \geq z\right) =\int_z^{+\infty}\frac{ x^{\al-1}e^{- x}}{\Gamma(\al)}\mathrm{d}x .$$
	For $\al,\beta,\lam > 0$ and $k>0$, the following equalities holds in distribution: \begin{align}\label{gamma_chi_relation}
\chi(k)= \sqrt{\Gamma(\frac{k}{2},\frac{1}{2})},\qquad  \lam \Gamma(\al,\beta)= \Gamma(\al,\frac{\beta}{\lam} ).
	\end{align}
\end{rmk}
The main step in our proof is the fulfillment of condition (\ref{condition_for_poisson3}) so we need a way to control summation of incomplete Gamma functions. Indeed, previous Remark \ref{gamma_chi_relation_rmk} indicates how chi's survival functions are related to incomplete Gamma functions. As we have to take limit on a sum, we require to work on a finite setting ($n<+\infty$) and estimate each summand by the crucial bound:
\begin{lem}\label{bounds_incomplete_lem}
	For $a<1$ and $z>0$, \begin{align}\label{bounds_incomplete_eq}
	\frac{z}{z+1-a}\frac{e^{-z}z^{a-1}}{\Gamma(a)} < \frac{\Gamma(a,z)}{\Gamma(a)}\leq \frac{e^{-z}z^{a-1}}{\Gamma(a)}.
	\end{align}
\end{lem}
\begin{proof}
	The upper bound is classic and can be found in \cite{Borwein_Chan}. One can use Pad\'e approximants to prove the lower bound of (\ref{bounds_incomplete_eq}). As a reference for this, we give \cite{nist}.
\end{proof}
We state two well-known facts which are going to be helpful to accuratly evaluate the whole sum.
\begin{rmk}
	It is well-known that the Gamma function admits the Taylor expansion near $u=0$: $$\Gamma(u)=\ff{u}+\gamma +\frac{6\gamma^2+\pi^2}{12} u+O(u^2).$$
	This combines well with the hypothesis $n\beta \ll 1$ since it implies $\ds{i\beta\xrightarrow[n\infty]{}0}$ uniformly in $i\leq n$, ie: $\ds{\max_{i\leq n}\left(i\beta \right)\xrightarrow[n\infty]{}0 }$. Thus,
	\begin{align}\label{Gamma_near_0}
	\Gamma\left( \frac{i\beta}{2} \right) &= \frac{2}{i\beta}+\gamma_{\text{Euler}} + O\left(n\beta\right).
	\end{align}
	For $u \in \R$, let us give another useful identity: \begin{align}\label{sum}
	\sum_{k=1}^{n}ku^k &= \frac{(nu-n-1)u^{n+1}+u}{(1-u)^2}.
	\end{align}
\end{rmk}
Lastly, we show a technical result which estimates an auxiliary term in the upcoming computations. Its outcome varies depending to $\beta$ regime and accurately shows the differences between the two regimes.
\begin{lem}\label{lambda_estimate}
	Let $z:=z_n\gg 1$ such that $\ds{z\ll \log^\delta(n)} $ for some $\delta>0$ and $\ds{\beta \ll \ff{n}}$. 
	
	We define: 	\begin{align*}
	\Lambda_n(z) &:= n\left(\left( \frac{z^2}{2}\right)^\frac{\beta}{2}  -1\right)  \left( \frac{z^2}{2}\right)^{\frac{n\beta}{2}}  -\left(\frac{z^2}{2} \right)^{\frac{n\beta}{2}} +1   . 
	\end{align*} Assume $\ds{\ff{n^2}\ll \beta \ll \ff{n\log \log \left( n\right) }  }$. Then, \begin{align}\label{lambda_first_case}\Lambda_n(z)&=
	\ff{2}\left( \frac{n\beta}{2}\log \left( \frac{z^2}{2} \right) \right)^2 \left(1+o(1) \right). 
	\end{align}
	
	Assume $\ds{\ff{n\log\log \left( n\right) }\ll \beta \ll \ff{n}  }$. Then, \begin{align}\label{lambda_second_case}\Lambda_n(z)&=
	\left(\frac{z^2}{2} \right)^{\frac{n\beta}{2}} \left( \frac{n\beta}{2} \log \left( \frac{z^2}{2}\right)\right) \left( 1+o(1)\right)  . 
	\end{align}
\end{lem}
\begin{proof}We begin with the first case $\ds{n^{-2}\ll \beta \ll \left( n\log \log n\right)^{-1}    }$. We treat each three main terms (except $n$) by second order Taylor expansion, so that $\Lambda_n(z)$ equals to:
	\begin{align*}
	&\left(\frac{n\beta}{2} \log \left( \frac{z^2}{2}\right)+\frac{n\left(\frac{\beta}{2} \log \left( \frac{z^2}{2}\right)\right)^2 }{2}  +nO\left( \beta^3 \log^3(z) \right)  \right) \times
	\\&\qquad \times  \left(1+ \frac{n\beta}{2} \log \left( \frac{z^2}{2}\right) + \frac{\left( \frac{n\beta}{2} \log \left( \frac{z^2}{2}\right)\right)^2 }{2}
	+ O\left( ( n\beta)^3 \log^3 \left( z\right)\right) \right)  
	\\&\quad -\frac{n\beta}{2}\log \left( \frac{z^2}{2} \right) - \frac{\left( \frac{n\beta}{2}\log \left( \frac{z^2}{2} \right) \right)^2 }{2}+O\left( n^3\beta^3\log^3 \left( z \right)  \right)      .
	\end{align*}
	Expanding the product, we get:
	\begin{align*}
	\Lambda_n(z)&=\frac{n\beta}{2} \log \left( \frac{z^2}{2}\right)+ \frac{n\left(\frac{\beta}{2} \log \left( \frac{z^2}{2}\right)\right)^2 }{2}   +nO\left( \beta^3 \log^3(z) \right)
	\\&+\left( \frac{n\beta}{2} \log \left( \frac{z^2}{2}\right)\right)^2 + \frac{n^2\left(\frac{\beta}{2} \right) ^3 \log^3 \left( \frac{z^2}{2}\right)}{2}+n^2 \beta \log(z)O\left( \beta^3 \log^3(z) \right) 
	\\&-\frac{n\beta}{2}\log \left( \frac{z^2}{2} \right) - \frac{\left( \frac{n\beta}{2}\log \left( \frac{z^2}{2} \right) \right)^2 }{2}+O\left( n^3\beta^3\log^3 \left( z   \right)\right) .
	\end{align*}
	The first order terms vanishes. Cleaning the highest order quantities, it leads to:
	$$\frac{n\left(\frac{\beta}{2} \log \left( \frac{z^2}{2}\right)\right)^2 }{2}   
	+ \frac{\left( \frac{n\beta}{2}\log \left( \frac{z^2}{2} \right) \right)^2 }{2}+O\left( n^3\beta^3\log^3 \left(z  \right)  \right)  .
	$$
	Thus, the result readily follows from the assumptions made on $z$ and $\beta$.
	
	Now, we assume $\ds{\ff{n\log\log \left( n\right) }\ll \beta \ll \ff{n}  }$. Likewise, 
	\begin{align*}
	\Lambda_n(z) &:= n\left(\left( \frac{z^2}{2}\right)^\frac{\beta}{2}  -1\right)  \left( \frac{z^2}{2}\right)^{\frac{n\beta}{2}}  -\left(\frac{z^2}{2} \right)^{\frac{n\beta}{2}} +1
	\\&= \left(\frac{z^2}{2} \right)^{\frac{n\beta}{2}} \left( n\left(\left(\frac{z^2}{2} \right)^{\frac{\beta}{2}} -1\right)   -1+\ff{\left(\frac{z^2}{2} \right)^{\frac{n\beta}{2}}}\right) 
	\\&= \left(\frac{z^2}{2} \right)^{\frac{n\beta}{2}} \left( n\left(\frac{\beta}{2} \log \left( \frac{z^2}{2}\right) +O\left(\beta^2 \log^2(z) \right)  \right)   -1+\ff{\left(\frac{z^2}{2} \right)^{\frac{n\beta}{2}}}\right) 
	\\&=\left(\frac{z^2}{2} \right)^{\frac{n\beta}{2}} \left( \frac{n\beta}{2} \log \left( \frac{z^2}{2}\right) +nO\left(\beta^2 \log^2(z) \right)     -1+\ff{\left(\frac{z^2}{2} \right)^{\frac{n\beta}{2}}}\right)
	\\&=\left(\frac{z^2}{2} \right)^{\frac{n\beta}{2}} \left( \frac{n\beta}{2} \log \left( \frac{z^2}{2}\right)\right) \left( 1+o(1)\right)  . 
	\end{align*}
\end{proof}
We can now turn to the actual proof of Theorem \ref{Main_th} which is in two stages accordingly to conditions (\ref{condition_for_poisson3}) and (\ref{condition_for_poisson4}). We prove the first condition.
\begin{Prop}\label{Main_prop}
	Let $\beta := \beta_n \ll 1$ such that $\ds{n^{-2}\ll \beta \ll n^{-1} }$ and $\left( X_{i,n}\right) $ a triangular array of independent random variables with $\ds{X_{i,n}\sim \chi \left( i\beta\right)  }$ for any $1\leq i \leq n$. 

	Assume $\ds{\ff{n^2}\ll \beta \ll \ff{n\log \log \left( n\right) }  }$.
	For $x\in\R^+$, we set: \begin{align}\label{def_scaling_1}
\vfi_n(x)&:=\frac{x}{a_n}+b_n= \frac{x}{\sqrt{2\log \left( n^2 \beta \right) }}+\sqrt{2\log \left( n^2 \beta \right) } - \frac{\log \log\left(  n^2 \beta\right)  }{\sqrt{2\log \left( n^2 \beta \right) }}.
	\end{align}
	Then for $0\leq x<+\infty$, $$S_n \left( \vfi_n(x)\right) :=\sum_{i=1}^{n}\p\left( X_{i,n} \geq \vfi_n(x)\right) \xrightarrow[n\infty]{} \frac{e^{-x}}{4}.$$
	Assume $\ds{ \ff{n\log \log \left( n\right) } \ll \beta \ll \ff{n} }$.
	For $x\in\R^+$, we set: 
	 \begin{align}\label{def_scaling_2}\vfi_n(x)&:= \frac{x}{\sqrt{2\log \left( n \right) }}+\sqrt{2\log \left( n \right) } - \frac{n\beta \log \log\left(  n\right)  }{\sqrt{2\log \left( n \right) }}-\frac{\log \log  \left( n\right) }{\sqrt{2\log \left( n \right) }}-\frac{\log \log \log \left( n\right) }{\sqrt{2\log \left( n \right) }}.\end{align} Then for $0\leq x<+\infty$, $$S_n \left( \vfi_n(x)\right) :=\sum_{i=1}^{n}\p\left( X_{i,n} \geq \vfi_n(x)\right) \xrightarrow[n\infty]{}e^{-x}.$$
\end{Prop}
\begin{proof}
	For convenience of notation, let $\ds{z:=\vfi_n(x)\gg 1}$. By Remark \ref{gamma_chi_relation_rmk}, we write the quantity of interest in terms of incomplete Gamma functions:
	\begin{align}\label{aux_equality}
	S_n(z)&=\sum_{i=1}^{n}\p\left(\chi^2\left( i\beta\right) \geq z^2 \right)  
	=\sum_{i=1}^{n}\p\left(\Gamma\left( \frac{i\beta}{2},1\right) \geq \frac{z^2}{2} \right)  
	= \sum_{i=1}^{n} \frac{\Gamma\left( \frac{i\beta}{2},\frac{z^2}{2}\right) }{\Gamma\left( \frac{i\beta}{2}\right) }.
	\end{align}
	By hypothesis $n\beta\ll 1$, we can apply Lemma \ref{bounds_incomplete_lem} to estimate (\ref{aux_equality}):
	\begin{align*}
	S_n(z)&\leq \sum_{i=1}^{n} \frac{e^{-\frac{z^2}{2}}z^{i\beta -2}}{\Gamma\left(\frac{i\beta}{2}\right) 2^{\frac{i\beta}{2}-1} }
	= \frac{2e^{-\frac{z^2}{2}}}{z^2}\sum_{i=1}^{n}\frac{\left( \frac{z^2}{2}\right)^{\frac{i\beta}{2}} }{\Gamma\left( \frac{i\beta}{2}\right) }.
		\end{align*}
		Applying the Gamma expansion (\ref{Gamma_near_0}) gives:
	\begin{align}\label{aux_upperbound2}
	S_n(z)&\leq 
	\frac{e^{-\frac{z^2}{2}}}{z^2}\sum_{i=1}^{n}\frac{i\beta \left( \frac{z^2}{2}\right)^{\frac{i\beta}{2}} }{1+\frac{i\beta}{2}\gamma_{\text{Euler}}  +  \frac{i\beta}{2}O\left( n\beta\right) }
	=  \frac{e^{-\frac{z^2}{2}}}{z^2}\frac{\beta}{1+O\left( n\beta\right) }\sum_{i=1}^{n}i\left(  \left(  \frac{z^2}{2}\right)^{\frac{\beta}{2}} \right)^i .
		\end{align}
		We analyze the last term with the identity (\ref{sum}):
		\begin{align*}
	\sum_{i=1}^{n}i\left(  \left(  \frac{z^2}{2}\right)^{\frac{\beta}{2}} \right)^i&=\frac{ \left( n\left( \frac{z^2}{2}\right) ^\frac{\beta}{2} -n-1\right) \left( \frac{z^2}{2}\right)^{\frac{n\beta}{2} + \frac{\beta}{2} }  + \left( \frac{z^2}{2}\right)^\frac{\beta}{2}  }{\left( 1-\left( \frac{z^2}{2}\right)^\frac{\beta}{2} \right)^2 }
	\\&= \ff{\left( 1-\left( \frac{z^2}{2}\right)^\frac{\beta}{2} \right)^2 }\left( \frac{z^2}{2}\right)^\frac{\beta}{2} \left(    \left( n\left( \frac{z^2}{2}\right) ^\frac{\beta}{2} -n-1\right) \left( \frac{z^2}{2}\right)^{\frac{n\beta}{2}  }  + 1  \right)    
	\\&=   \ff{\left( 1-\left( \frac{z^2}{2}\right)^\frac{\beta}{2} \right)^2 }\left( \frac{z^2}{2}\right)^\frac{\beta}{2} \left( n\left(\left( \frac{z^2}{2}\right)^\frac{\beta}{2}  -1\right)  \left( \frac{z^2}{2}\right)^{\frac{n\beta}{2}}  +  1-\left(\frac{z^2}{2} \right)^{\frac{n\beta}{2}}    \right)    
	\\&=  \frac{	 n\left(\left( \frac{z^2}{2}\right)^\frac{\beta}{2}  -1\right)  \left( \frac{z^2}{2}\right)^{\frac{n\beta}{2}}  -\left(\frac{z^2}{2} \right)^{\frac{n\beta}{2}} +1  }{\left( 1-\left( \frac{z^2}{2}\right)^\frac{\beta}{2} \right)^2 }\left( \frac{z^2}{2}\right)^\frac{\beta}{2}      .
	\end{align*}
	The decay hypothesis on $\beta$ and definitions (\ref{def_scaling_1}), (\ref{def_scaling_2}) imply that $\ds{\beta\log(z)\ll 1}$ and hence, $$\ff{\left( 1-\left( \frac{z^2}{2}\right)^\frac{\beta}{2} \right)^2 }=\ff{\left( \frac{\beta}{2} \log \left(\frac{z^2}{2} \right)  \right)^2}\left( 1+o(1)\right), \qquad\left( \frac{z^2}{2}\right)^\frac{\beta}{2}=1+o(1) .$$
	Combining everything in (\ref{aux_upperbound2}), we get: $$S_n(z)\leq  \frac{4e^{-\frac{z^2}{2}}}{z^2} \frac{\Lambda_n(z)}{\beta \log^2(\frac{z^2}{2})}\left( 1+o(1)\right)   ,$$ with $$	\Lambda_n(z) :=n\left(\left( \frac{z^2}{2}\right)^\frac{\beta}{2}  -1\right)  \left( \frac{z^2}{2}\right)^{\frac{n\beta}{2}}  -\left(\frac{z^2}{2} \right)^{\frac{n\beta}{2}} +1 . $$
	
	We turn to the lower bound. For any $\ds{a\in [0,1]}$ and $z\gg 1$, we have: \begin{align}\label{rmk_on_lowerbound}
	\frac{z}{z+1-a}&\geq \frac{z}{z+1} = 1+o(1).
	\end{align} Applying the lower bound (\ref{bounds_incomplete_eq}) on (\ref{aux_equality}) and according to (\ref{rmk_on_lowerbound}), we get:
	\begin{align*}
	S_n(z)&\geq \sum_{i=1}^{n} \frac{z}{z+1-\frac{i\beta}{2}}\frac{e^{-\frac{z^2}{2}}z^{i\beta -2}}{\Gamma\left(\frac{i\beta}{2}\right) 2^{\frac{i\beta}{2}-1} }
	\\&=\left(1+o(1) \right) \frac{2e^{-\frac{z^2}{2}}}{z^2}\sum_{i=1}^{n}\frac{\left( \frac{z^2}{2}\right)^{\frac{i\beta}{2}} }{\Gamma\left( \frac{i\beta}{2}\right) }.
	\end{align*}
	Repeating the same steps as before, it yields: \begin{align}\label{intermediate_equivalent}
S_n(z)&=   \frac{4e^{-\frac{z^2}{2}}}{z^2} \frac{\Lambda_n(z)}{\beta \log^2(\frac{z^2}{2})}\left( 1+o(1)\right)    .
	\end{align}
	Note that we have not made any additional assumptions on $\ds{\beta\ll n^{-1}}$ yet. The next step is to estimate the $\Lambda_n(z)$ term. According to the $\beta$ regime considered, two different outcomes turn out, which is described in Lemma \ref{lambda_estimate}. 
	
	We begin with the first case and restrict to $\ds{n\beta \log z \ll 1}$, ie: $\ds{ n\beta \log\log n \ll 1 }$. Combining the asymptotic (\ref{lambda_first_case}) with (\ref{intermediate_equivalent}), it follows that:
	\begin{align}\label{last_step1}
	S_n(z)& 
	= \exp\left( -\frac{z^2}{2} +\log(n^2\beta )-2\log (z) -\log 2\right)\left(1+o(1) \right) .
	\end{align}
	From the definition (\ref{def_scaling_1}) of the scaling function, we compute the following asymptotics:
	$$-\frac{\vfi_n(x)^2}{2}= -x-\log \left( n^2 \beta \right) + \log \log (n^2\beta)+o(1)$$
	$$ \log \vfi_n(x) =\log \left( \sqrt{2\log \left( n^2 \beta \right) } \left( 1+o(1)\right) \right) = \ff{2} \log \log \left( n^2 \beta \right) +\ff{2}\log 2 +o(1).$$
	And finally, substituting in (\ref{last_step1}), the result yields:
	$$S_n(z) = \frac{e^{-x}}{4}(1+o(1)). $$
	
	We consider the second regime case for $\beta$ and assume in the sequel that $\ds{ \ff{n\log \log \left( n\right) } \ll \beta \ll \ff{n} }$.
	The main difference lies in Lemma \ref{lambda_estimate}. We substitute the asymptotic (\ref{lambda_second_case}) in (\ref{intermediate_equivalent}): \begin{align}\label{last_step2}
	S_n(z)&=  4\exp\left( -\frac{z^2}{2} +n\beta \log(z)-2\log (z) - \log\log(z) +\log (n)\right)\left(1+o(1) \right) .
	\end{align}
	Again, we lastly compute from (\ref{def_scaling_2}) the required asymptotics:
	\begin{align}\label{asymptotic1}
-\frac{z^2}{2}& = -x-\log(n) -n\beta \log\log(n) +\log\log(n)+\log\log\log(n)+o(1) 
	\end{align}
	\begin{align}\label{asymptotic2}
	\log(z)&=\log \left(\sqrt{2\log(n)}(1+o(1)) \right)=\ff{2}\log\log(n)+\ff{2}\log(2)+o(1) 
	\end{align}
	$$-\log\log(z) = -\log\log\log(n)-\log(2)+o(1).$$
	Combining these with (\ref{last_step2}), it leads to the desired conclusion:
	$$S_n(z)=e^{-x}\left( 1+o(1)\right) .$$
\end{proof}
Finally, we show condition (\ref{condition_for_poisson4}):
\begin{lem}\label{uniforme_negligeabilite_nbeta_petit}
Let $0\leq y<x<+\infty$. Under the assumptions of Proposition \ref{Main_prop}, one has: \begin{align}\label{uniform_negl_limit}
\max_{i\leq n}\left|\p\left(X_{i,n}\geq \vfi_n(x) \right) - \p\left(X_{i,n}\geq \vfi_n(y) \right)\right| &\xrightarrow[n\infty]{}0.\end{align}
\end{lem}
\begin{proof}
For $\vfi_n$ such as in (\ref{def_scaling_1}) or (\ref{def_scaling_2}), and $i\leq n$, let us write: $$\xi_{i,n}:=  \left| \p\left(X_{i,n}\geq \vfi_n(x) \right) - \p\left(X_{i,n}\geq \vfi_n(y) \right)\right|.$$
There exists a constant $c>0$ such that:\begin{align}\label{upperbound_for_negl}
\xi_{i,n}&= \int_{\vfi_n(y)}^{\vfi_n(x)} \frac{t^{i\beta-2}e^{-\frac{t^2}{2}}}{2^{\frac{i\beta}{2}-1}\Gamma(\frac{i\beta}{2})} \mathrm{d}t
\leq \frac{\vfi_n(x)^{i\beta-2}e^{-\frac{\vfi_n(y)^2}{2}}}{2^{\frac{i\beta}{2}-1}\Gamma(\frac{i\beta}{2})}
\leq c \frac{\vfi_n(x)^{n\beta}e^{-\frac{\vfi_n(y)^2}{2}}}{\vfi_n(x)^2\Gamma(\frac{i\beta}{2})}.
\end{align}

Assume that $\ds{ \left( n\log \log \left( n\right) \right)^{-1}  \ll \beta \ll n^{-1}}$ and $\vfi_n$ as in (\ref{def_scaling_2}). In this case, considering the asymptotics (\ref{asymptotic1}), (\ref{asymptotic2}) and the fact that \begin{align}\label{asymptotic_rmk}
\vfi_n \gg 1, \qquad \min_{i\leq n}  \Gamma\left(\frac{i\beta}{2} \right) \gg 1,
\end{align} a little computation shows that the latter quantity in (\ref{upperbound_for_negl}) converges to $0$ uniformly in $i\leq n$ as $n\to +\infty$.

When $\ds{n^{-2}\ll \beta \ll \left( n\log\log n\right) ^{-1}  }$ and $\vfi_n$ as in (\ref{def_scaling_1}), the conclusion follows from (\ref{asymptotic_rmk}) and $\ds{\vfi_n(x)^{n\beta}=1+o(1)}$.
\end{proof}
\section{Proof of Theorem \ref{Main_th2}}\label{section_proof2}
From the normalization sequences $(a_n),(b_n)$ in Theorem \ref{Main_th2}, we define the scaling function $\vfi_n$ in the same way that (\ref{def_scaling_1}) by: \begin{align}\label{def_vfi_gamma}
\vfi_n(x):=\frac{x}{a_n}+b_n, \qquad x\in \R^+.
\end{align} This is an increasing homeomorphism on $\R^+$ with inverse $\ds{\vfi^{-1}_n(x)=a_n\left(x-b_n\right) }$. We note that it corresponds to the scaling function of $\ds{\frac{n}{\gamma \log\log(n)}}$ i.i.d. copies of $\chi(2\gamma)$. In other words, for $x<+\infty$, it satisfies: $$\frac{n}{\gamma\log \log(n)}\p\big(X_{n,n}\geq \vfi_n(x)\big) \xrightarrow[n\infty]{} e^{-x}.$$
Likewise Theorem \ref{Main_th}, we follow the path of Theorem \ref{Th_tool} and begin with the assumption (\ref{condition_for_poisson3}). We control the sum of chi's survival functions with the same approach. The regime $\ds{n\beta =2\gamma}$ does not allow an uniform negligibility on the chi's random variables. In Theorem \ref{Main_th}, one takes into consideration all summands, each underlying random variables contributing to the extreme value.
 This time, the $X_{i,n}$ terms with index $i$ negligible to $n$, ie: $i\ll n$, are inconsequential. In order to get meaninful at the limit, we need to explore the region of $n$. We cut off at the appropriate section of the sum, hence identifying the significant stack of chi's random variables significant to the maximum.
\begin{lem}\label{lem_exp_limit_gamma} Let $\ds{\gamma \in (0,1)}$. Let $\beta := \beta_n \ll 1$ such that $\ds{n\beta ={2\gamma} }$ and $\ds{X_{i,n}\sim \chi \left( i\beta\right)  }$, with $1\leq i \leq n$, a triangular array of independent random variables and let $\vfi_n$ defined in (\ref{def_vfi_gamma}). 
	
	Then for $0\leq x<+\infty$,
	$$\sum_{i=1}^{n} \p \big(X_{i,n}\geq \vfi_n(x)\big)\xrightarrow[n\infty]{}e^{-x}. $$
\end{lem}
\begin{proof} 
	For clarity, we will work with $z$ instead of $\vfi_n(x)$, $x\in\R^+$. By Remark \ref{gamma_chi_relation_rmk}, 
		\begin{align}\label{aux_equality2}
\sum_{i=1}^{n}\p\left(\chi\left( i\beta\right) \geq z \right)  
	=\sum_{i=1}^{n}\p\left(\Gamma\left( \frac{i\beta}{2},1\right) \geq \frac{z^2}{2} \right)  
	= \sum_{i=1}^{n} \frac{\Gamma\left( \frac{i\beta}{2},\frac{z^2}{2}\right) }{\Gamma\left( \frac{i\beta}{2}\right) }=\sum_{i=1}^{n} \frac{\Gamma\left(\gamma \frac{i}{n},\frac{z^2}{2}\right) }{\Gamma\left(\gamma \frac{i}{n}\right) }.
	\end{align}
	Since $\ds{i\frac{\gamma}{n}<1}$ for any $i\leq n$, one can estimate (\ref{aux_equality2}) with Lemma \ref{bounds_incomplete_lem}:	$$\frac{z^2}{z^2+2}e^{-\frac{z^2}{2}}\sum_{i=1}^{n}\frac{z^{\gamma\frac{i}{n}-2}}{2^{\frac{i\gamma}{2n}-1}\Gamma(\gamma\frac{i}{n})} <\sum_{i=1}^{n}\p\big(X_{i,n}\geq z\big)\leq e^{-\frac{z^2}{2}}\sum_{i=1}^{n}\frac{z^{\gamma\frac{i}{n}-2}}{2^{\frac{i\gamma}{2n}-1}\Gamma(\gamma\frac{i}{n})}.$$
	The extra-term $\ds{\frac{z^2}{z^2+2}}$ in the lower bound converges to $1$ because $z\gg 1$. Hence, it is sufficient to prove that $\ds{S_n(z):=e^{-\frac{z^2}{2}}\sum_{i=1}^{n}\frac{z^{\gamma\frac{i}{n}-2}}{2^{\frac{i\gamma}{2n}-1}\Gamma(\gamma\frac{i}{n})}}$ converges to $\ds{e^{-x}}$.
	
	For this purpose, let $\mu :=\mu_n$ such that $\ds{\ff{\log\log(n)}\ll \mu_n \ll 1}$, e.g., $\ds{\mu_n =\big(  \log \log (n) \big)^{-\ff{2}} }$.
	We split $\ds{S_n(z) =S^1_n(z)+S^2_n(z)}$ with $$S^1_n(z):= e^{-\frac{z^2}{2}}\sum_{i=1}^{n(1-\mu)}\frac{z^{\gamma\frac{i}{n}-2}}{2^{\frac{i\gamma}{2n}-1}\Gamma(\gamma\frac{i}{n})},\qquad S^2_n(z):= e^{-\frac{z^2}{2}}\sum_{i=n(1-\mu)}^{n}\frac{z^{\gamma\frac{i}{n}-2}}{2^{\frac{i\gamma}{2n}-1}\Gamma(\gamma\frac{i}{n})} ,$$ and show that: $$S^1_n(z) \xrightarrow[n\infty]{} 0,\qquad S^2_n(z) \xrightarrow[n\infty]{}e^{-x}.$$
	We begin with the main core term $S^2_n(z)$. Let us define: \begin{align}\label{m_n_M_n2}
	{m_n}^{-1} := \min_{n(1-\mu)\leq i \leq n}2^{\frac{i\gamma}{2n}} \Gamma(\gamma\frac{i}{n})\sim 2^{\frac{\gamma}{2}}\Gamma(\gamma),\quad {M_n}^{-1} := \max_{n(1-\mu)\leq i \leq n}2^{\frac{i\gamma}{2n}} \Gamma(\gamma\frac{i}{n})\sim 2^{\frac{\gamma}{2}}\Gamma(\gamma) .
	\end{align}
	Consequently, \begin{align}\label{consequently}
M_n\frac{2e^{-\frac{z^2}{2} }}{z^2}\sum_{i=n(1-\mu)}^{n}\left(\frac{z^2}{2} \right) ^{\gamma\frac{i}{2n}}\leq S^2_n(z) \leq m_n\frac{2e^{-\frac{z^2}{2} }}{z^2}\sum_{i=n(1-\mu)}^{n}\left(\frac{z^2}{2} \right) ^{\gamma\frac{i}{2n}}.
	\end{align}
	We compute the geometric sum in the previous line:
	\begin{align}\label{compute_geometric_sum}
\frac{2e^{-\frac{z^2}{2} }}{z^2}\sum_{i=n(1-\mu)}^{n}\left(\frac{z^2}{2} \right) ^{\gamma\frac{i}{2n}}= e^{-\frac{z^2}{2} }\frac{\left( \frac{z^2}{2}\right) ^{\frac{\gamma}{2}-1} \left( 1-\left(\frac{z^2}{2} \right)^{-\frac{\gamma \mu}{2}}\right) }{\left( \frac{z^2}{2}\right)^{\frac{\gamma}{2n}}-1} .
	\end{align}
	 In order to estimate (\ref{compute_geometric_sum}), we use Taylor expansion, the hypothesis on $\mu$, the definition (\ref{def_vfi_gamma}) and (\ref{m_n_M_n2}), which provide the asymptotics: \begin{align}\label{intermediate_asymptotics_gamma}
 \left(\frac{z^2}{2} \right)^{-\frac{\gamma \mu}{2}}\ll 1,\qquad \left( \frac{z^2}{2}\right)^{\frac{\gamma}{2n}}-1 =  \frac{\gamma}{2n}\log\log(n)\left( 1+o(1) \right)
	 \end{align}
	Combining (\ref{m_n_M_n2}), (\ref{compute_geometric_sum}) and (\ref{intermediate_asymptotics_gamma}) with the bound (\ref{consequently}), it provides the asymptotic of $\ds{S^2_n(z)}$: \begin{align}\label{S^2_n(z)2} 
	\exp\left(-\frac{z^2}{2}+\left( \frac{\gamma}{2}-1\right) \log\left( \frac{z^2}{2}\right) +\log(n)-\log^{(3)}(n)-\log\left(\frac{2^{\frac{\gamma}{2}-1}\gamma}{\Gamma(\gamma)} \right)+o(1)\right) .
	\end{align}
	Finally, together with (\ref{def_vfi_gamma}) and the fact that
 \begin{align}\label{log(z^2/2)}
	\left( \frac{\gamma}{2}-1\right) \log\left( \frac{z^2}{2}\right)=\left( \frac{\gamma}{2}-1\right) \log\log(n) +\left( \frac{\gamma}{2}-1\right) \log(2)+o(1)
 \end{align}
	\begin{align}\label{-z^2/2}
-\frac{z^2}{2} =-x-\log (n)+\log^{(3)}(n)-\left( \frac{\gamma}{2}-1\right) \log\log({n})+\log\left(\gamma 2^{-\frac{\gamma}{2}}\Gamma(\gamma)\right)+o(1),
	\end{align}	
	the expression (\ref{S^2_n(z)2}) yields the claim.

	Regarding the first sum $S^1_n(z)$, there exists a numerical constant $c>0$ such that: $$ S^1_n(z)=e^{-\frac{z^2}{2}}\sum_{i=1}^{n(1-\mu)}\frac{z^{\gamma\frac{i}{n}-2}}{2^{\frac{i\gamma}{2n}-1}\Gamma(\gamma\frac{i}{n})} \leq c\frac{2e^{-\frac{z^2}{2} }}{z^2}\sum_{i=1}^{n(1-\mu)}\left(\frac{z^2}{2} \right) ^{\frac{i\gamma}{2n}} = c\frac{2e^{-\frac{z^2}{2}}}{z^2} \frac{\left( \frac{z^2}{2} \right) ^{ \frac{\gamma}{2}(1-\mu)}-\left( \frac{z^2}{2} \right) ^{\frac{\gamma}{2n}}}{\left(\frac{z^2}{2} \right) ^{\frac{\gamma}{2n}}-1}.$$
	By Taylor expansion, we find another constant $c'>0$ such that the previous upper bound asymptoticaly equals to:
	$$c'e^{-\frac{z^2}{2}} \frac{2n}{\gamma\log\log\left( \frac{z^2}{2}\right) }e^{\left( \frac{ \gamma}{2}-1-\frac{\gamma}{2}\mu\right) \log\left( \frac{z^2}{2}\right)  }\left( 1+o(1)\right).$$ Using (\ref{log(z^2/2)}) and (\ref{-z^2/2}), for some constant $c''>0$, we have:
	$$0\leq  S^1_n(z)\leq c''\exp\left( -\frac{\gamma}{2}\mu \log\log(n)  +o(1)\right) .$$
	The hypothesis on $\mu$ is precisely that $\mu \log\log(n)\gg 1$, thus the proof is complete.
\end{proof}
To conclude the proof, it only remains to show:
\begin{lem}
	Given the assumptions of Theorem \ref{Main_th2} and $\vfi_n$ defined in (\ref{def_vfi_gamma}), for any $0\leq x<y<+\infty$,  one has: $$ \max_{i\leq n}\p\big(\vfi_n(x)\leq X_{i,n}\leq \vfi_n(y)\big)\xrightarrow[n\infty]{}0.$$ 
\end{lem}
\begin{proof}
	Let $x<y$ and $i\leq n$, then \begin{align*}
	\p\big(\vfi_n(x)\leq X_{i,n}\leq \vfi_n(y)\big)&= \int_{\vfi_n(x)}^{\vfi_n(y)} \frac{e^{-u}u^{t\frac{i}{n}-1}}{\Gamma(t\frac{i}{n})}du\\&\leq \ff{\min_{\R^+} \Gamma}\big(\vfi_n(y)-\vfi_n(x)\big) e^{-\vfi_n(x)} \vfi_n(y)^{t-1}
	\\ &=\ff{\min_{\R^+} \Gamma}(y-x) e^{-\vfi_n(x)} \vfi_n(y)^{t-1}
	\end{align*}
	The latter term is independent of $i\leq n$ and converges to $0$.
\end{proof}
		\bibliographystyle{plain}
	\bibliography{refextreme}
	\Addresses
\end{document}